\documentclass[11pt]{article}
\usepackage[a4paper,margin=1.1in]{geometry}

\usepackage[colorlinks,citecolor=magenta,linkcolor=black]{hyperref}
\pdfpagewidth=\paperwidth \pdfpageheight=\paperheight
\usepackage{amsfonts,amssymb,amsthm,amsmath,eucal,tabu,url}
\usepackage{pgf}
 \usepackage{array}
 \usepackage{pstricks}
 \usepackage{pstricks-add}
 \usepackage{pgf,tikz}
 \usetikzlibrary{automata}
 \usetikzlibrary{arrows}
 \usepackage{indentfirst}
\usepackage{tabularx} 

\theoremstyle{plain}
\newtheorem{thm}{Theorem}[section]
\newtheorem{theorem}[thm]{Theorem}

\newtheorem{proposition}[thm]{Proposition}
\newtheorem{corollary}[thm]{Corollary}

\theoremstyle{definition}
\newtheorem{definition}[thm]{Definition}
\newtheorem{remark}[thm]{Remark}

\newtheorem{question}[thm]{Question}

\newtheorem{thevarthm}[thm]{\varthmname}

\newenvironment{varthm*}[1]{\trivlist\item[]{\bf #1.}\it}{\endtrivlist}

\renewcommand\geq{\geqslant}

\renewcommand\leq{\leqslant}

\newcommand\be{\begin{eqnarray*}}
\newcommand\ee{\end{eqnarray*}}

\newcommand\newop[2]{\def#1{\mathop{\rm #2}\nolimits}}
\newop\log{log}
\newop\ord{ord}
\newop\Gal{Gal}
\newop\SL{SL}
\newop\Bl{Bl}
\newop\mult{mult}
\newop\mass{mass}
\newop\div{div}
\newop\codim{codim}
\newop\sing{sing}
\newop\vdim{vdim}
\newop\edim{edim}
\newop\Ass{Ass}
\newop\size{size}
\newop\reg{reg}
\newop\satdeg{satdeg}
\newop\supp{supp}
\newop\Neg{Neg}
\newop\Nef{Nef}
\newop\Nefh{Nef_H}
\newop\Eff{Eff}
\newop\Zar{Zar}
\newop\MB{MB}
\newop\MBxC{MB\mathit{(x,C)}}
\newop\NnB{NnB}
\newop\Bigg{Big}
\newop\Effbar{\overline{\Eff}}

\def\keywordname{{\bfseries Keywords}}%
\def\keywords#1{\par\addvspace\medskipamount{\rightskip=0pt plus1cm
\def\and{\ifhmode\unskip\nobreak\fi\ $\cdot$
}\noindent\keywordname\enspace\ignorespaces#1\par}}
\def\subclassname{{\bfseries Mathematics Subject Classification
(2020)}\enspace}
\def\subclass#1{\par\addvspace\medskipamount{\rightskip=0pt plus1cm
\def\and{\ifhmode\unskip\nobreak\fi\ $\cdot$
}\noindent\subclassname\ignorespaces#1\par}}

\begin{document}
\title{A note on free determinantal hypersurface arrangements in $\mathbb{P}^{14}_{\mathbb{C}}$}
\author{Marek Janasz and Paulina Wi\'sniewska}
\date{\today}
\maketitle
\thispagestyle{empty}
\begin{abstract}
In the present note we study determinantal arrangements constructed with use of the $3$-minors of a $3 \times 5$ generic matrix of indeterminates. In particular, we show that certain naturally constructed hypersurface arrangements in $\mathbb{P}^{14}_{\mathbb{C}}$ are free.
\subclass{14N20, 14C20}
\end{abstract}
\section{Introduction}
The main aim of the present note is to find new examples of free hypersurfaces arrangements constructed as the so-called determinantal arrangements. These arrangements possess many interesting homological property and some of them will be outlined. On the other side, computations related to these arrangements are very involving and probably this is the main reason why these object are not well-studied yet. In the note we focus on determinantal arrangements constructed via the $3$ minors of a $3\times 5$ generic matrix.
Before we present our main results, let us summarize briefly the basic concepts (see \cite{Dimca,ScTo} for more details). 

Let $\mathcal{C} \subset \mathbb{P}^{n}$ be an arrangement of reduced and irreducible hypersurfaces and let $\mathcal{C} = V(F)$, where $F = f_{1} \cdots f_{d}$ with ${\rm GCD}(f_{i},f_{j}) = 1$. Denote by ${\rm Der}(S) = S \cdot \partial_{x_{0}} \oplus ... \oplus S\cdot \partial_{x_{n}}$ the ring of polynomial derivations, where $S = \mathbb{K}[x_{0}, ..., x_{n}]$ and $\mathbb{K}$ is a field of characteristic zero. If we take $\theta \in {\rm Der}(S)$, then
$$\theta(f_{1} \cdots f_{d}) = f_{1} \cdot \theta(f_{2} \cdots f_{d}) + f_{2} \cdots f_{d} \cdot \theta(f_{1}).$$ 
Now we can define the ring of polynomial derivations tangent to $\mathcal{C}$ as
$$D(\mathcal{C}) = \{ \theta \in {\rm Der}(S) \, : \, \theta(F) \in F \cdot S \}.$$
An inductive application of the Leibniz formula leads us to the following characterization of $D(\mathcal{C})$, namely
$$D(\mathcal{C}) = \{ \theta \in {\rm Der}(S) \, : \, \theta(f_{i}) \in f_{i} \cdot S \text{ for } i \in \{1, ..., d\}\}.$$
We have the following (automatic) decomposition 
$$D(\mathcal{C}) \simeq E \oplus D_{0}(\mathcal{C}),$$
where $E$ is the Euler derivation and $D_{0}(\mathcal{C}) = {\rm syz}(J_{F})$ is the module of syzygies for the Jacobian ideal $J_{F} = \langle \partial_{x_{0}}F, ..., \partial_{x_{n}}F \rangle$ of the defining polynomial $F$. 
The freeness of $\mathcal{C}$ boils down to a question whether ${\rm pdim}(S/J_{F}) = 2$, which is equivalent to $J_{F}$ being Cohen-Macaulay. One can show that a reduced hypersufrface $\mathcal{C} \subset \mathbb{P}^{n}$ given by a homogeneous polynomial $F=0$ is free if the following condition holds: the minimal resolution of the Milnor algebra $M(F) = S / J_{F}$ has the following short form
$$0 \rightarrow \bigoplus_{i=1}^{n}S(-d_{i} -(d-1)) \rightarrow S^{n+1}(-d+1) \rightarrow S,$$
and the multiset of integers $(d_{1}, ..., d_{n})$ with $d_{1} \leq  ... \leq d_{n}$ is called the set of exponents of $D_{0}(\mathcal{C})$, and we will denote it by ${\rm exp}(\mathcal{C})$.

The literature devoted to determinantal arrangements is not robust. In this context it is worth recalling a general result by  Yim \cite[Theorem 3.3]{Yim}, where he is focusing on determinantal arrangements in $\mathbb{P}^{2n-1}_{\mathbb{C}}$ defined by the products of the $2$-minors. For $i<j$ we denote the $2$-minor of the matrix 
$$N = \begin{pmatrix}
x_{1} & x_{2} & x_{3} & ... & x_{n} \\
y_{1} & y_{2} & y_{3} & ... & y_{n} 
\end{pmatrix}$$
by $\triangle_{ij} = x_{i}y_{j} - x_{j}y_{i}$. Consider arrangement $\mathcal{A}$ defined by the polynomial $F = \prod_{1 \leq i < j \leq n} \triangle_{ij}$ with $n\geq 3$. Then the arrangement $\mathcal{A}$ is free and a basis of $D(\mathcal{A})$ can be very explicitly described.

Our research is motivated by the following question \cite[Question 3.4]{Yim}.
\begin{question}
Let $M$ be the $m\times n$ matrix of indeterminates, and let $F$ be the product of all maximal minors of $M$. Is the arrangement defined by $F$ free for any $n>m>2$?
\end{question}
\begin{remark}
First of all, if $\mathcal{C} : F=0$ is the hypersurface defined by the determinant of a generic $3 \times 3$ matrix of indeterminates, then $\mathcal{C}$ is far away from being free. Buchweitz and Mond  in \cite{Buch} showed that the arrangement defined by the product of the maximal minors of a generic $n \times (n+1)$ matrix of indeterminates is free (and it means that we have the freeness property when $m=3$ and $n=4$), so the first non-trivial and unsolved case (to the best of our knowledge) is when $m=3$ and $n=5$.
\end{remark}
Let us consider the $3 \times 5$ matrix of indeterminates
$$M = \begin{pmatrix}
x_{1} & x_{2} & x_{3} & x_{4} & x_{5} \\
y_{1} & y_{2} & y_{3} & y_{4} & y_{5} \\
z_{1} & z_{2} & z_{3} & z_{4} & z_{5}
\end{pmatrix}.$$
Now for a triple $\{i,j,k\}$ with $i < j < k$ we construct the $3$-minor of $M$ by taking $i$-th, $j$-th, and $k$-th column.
Using the $3$-minors we can get $10$ hypersurfaces $H_{l}  \subset \mathbb{P}^{14}$ which are given by the following defining polynomials:
$$f_{1} = -x_3 y_2 z_1 + x_2 y_3 z_1 + x_3 y_1 z_2 - x_1 y_3 z_2 - x_2 y_1 z_3 + x_1 y_2 z_3,$$
$$f_{2} = -x_4 y_2 z_1 + x_2 y_4 z_1 + x_4 y_1 z_2 - x_1 y_4 z_2 - x_2 y_1 z_4 + x_1 y_2 z_4,$$
$$f_{3} = -x_4 y_3 z_1 + x_3 y_4 z_1 + x_4 y_1 z_3 - x_1 y_4 z_3 - x_3 y_1 z_4 + x_1 y_3 z_4,$$
$$f_{4} = -x_4 y_3 z_2 + x_3 y_4 z_2 + x_4 y_2 z_3 - x_2 y_4 z_3 - x_3 y_2 z_4 + x_2 y_3 z_4,$$ 
$$f_{5} = -x_5 y_2 z_1 + x_2 y_5 z_1 + x_5 y_1 z_2 - x_1 y_5 z_2 - x_2 y_1 z_5 + x_1 y_2 z_5,$$
$$f_{6} = -x_5 y_3 z_1 + x_3 y_5 z_1 + x_5 y_1 z_3 - x_1 y_5 z_3 - x_3 y_1 z_5 + x_1 y_3 z_5,$$
$$f_{7} = -x_5 y_3 z_2 + x_3 y_5 z_2 + x_5 y_2 z_3 - x_2 y_5 z_3 - x_3 y_2 z_5 + x_2 y_3 z_5,$$
$$f_{8} = -x_5 y_4 z_1 + x_4 y_5 z_1 + x_5 y_1 z_4 - x_1 y_5 z_4 - x_4 y_1 z_5 + x_1 y_4 z_5,$$
$$f_{9} = -x_5 y_4 z_2 + x_4 y_5 z_2 + x_5 y_2 z_4 - x_2 y_5 z_4 - x_4 y_2 z_5 + x_2 y_4 z_5,$$
$$f_{10}= -x_5 y_4 z_3 + x_4 y_5 z_3 + x_5 y_3 z_4 - x_3 y_5 z_4 - x_4 y_3 z_5 + x_3 y_4 z_5.$$

Using these $3$-minors we would like to explore new examples of free divisors constructed as determinantal arrangements of hypersurfaces.

In order to show the freeness of such arrangements, we are going to use the following criterion due to Saito -- see for instance \cite[Theorem 8.1]{Dimca}.
Let $\mathcal{C} \subset \mathbb{P}^{n}$ be a reduced effective divisor defined by a homogeneous equation $f = 0$. Now we define the graded module of all Jacobian syzygies as
$${\rm AR}(f) := \bigg\{ r = (a_{0}, ...,a_{n}) \in S^{n+1} \, : \, a_{0}\cdot \partial_{x_{0}}(f) + ... + a_{n} \cdot \partial_{x_{n}}(f) = 0\bigg\}.$$
To each Jacobian relation $r \in {\rm AR}(f)$ one can associate a derivation 
$$D(r) = a_{0}\cdot \, \partial_{x_{0}} + ... + a_{n}\cdot \, \partial_{x_{n}}$$
that kills $f$, i.e., $D(r)(f)=0$. One can additionally show that in fact ${\rm AR}(f)$ is isomorphic, as a graded $S$-module, with $D_{0}(\mathcal{C})$.

\begin{theorem}

The homogeneous Jacobian syzygies $r_{i} \in {\rm AR}(f)$ for $i \in \{1, ...,n\}$ form a basis of this $S$-module if and only if
$$\phi(f) = c \cdot f,$$
where $\phi(f)$ is the determinant of the $(n+1)\times (n+1)$ matrix $\Phi(f) = (r_{ij})_{0 \leq i,j \leq n}$, $r_{0} := (x_{0}, ..., x_{n})$, and $c$ is a non-zero constant.
\end{theorem}
Saito's criterion is a very powerful tool under the assumption that we have a set of potential candidates that might form a basis of ${\rm AR}(f)$, so our work boils down to finding appropriate sets of Jacobian relations that will lead us to a basis of ${\rm AR}(f)$ for a given arrangement $\mathcal{C} : \, f=0$.

Here is our first result of the note.

\begin{theorem}
\label{theoremA}
Let us consider the following hypersurfaces arrangements
$$\mathcal{C}_{j} : F_{j} = f_{1}f_{2}f_{3}f_{4}f_{j} \quad \text{ for } j\in \{5, ...,10\}.$$
Then $C_{j}$ is free with the exponents $(\underbrace{1, ..., 1}_{14 \text{ times}})$.
\end{theorem}
\begin{corollary}
In the setting of the above theorem, one has
$${\rm reg}(S/J_{F_{j}}) = 13$$
for each $j \in \{5, ..., 10\}$, so we reach an upper bound for the regularity according to the content of \cite[Proposition 2.6]{BSDS}.
\end{corollary}
\begin{remark}
Of course not every combination of $5$ defining equations $f_{i},f_{j},f_{k},f_{l},f_{m}$ leads to an example of a free determinantal arrangement. Consider $\mathcal{A}: \, V(f_{1}f_{2}f_{3}f_{5}f_{10})=0$, then the minimal free resolution of the Milnor algebra $M(F)=S/J_{F}$ with $F = f_{1}f_{2}f_{3}f_{5}f_{10}$ has the following form:
$$0 \rightarrow S(-19)^{3} \rightarrow S^{4}(-18) \oplus S^{13}(-15) \rightarrow S^{15}(-14)\rightarrow S,$$
so the projective dimension is equal to $3$.

Moreover, not every choice of $5$ \emph{consecutive} hyperplanes leads to a free arrangement. Consider $\mathcal{B} \, : V(f_{6}f_{7}f_{8}f_{9}f_{10}) = 0$, then the minimal free resolution of the Milnor algebra has the following form
$$0 \rightarrow S(-16)^{3} \rightarrow S^{1}(-18) \oplus S^{16}(-15) \rightarrow S^{15}(-14)\rightarrow S,$$
so $\mathcal{B}$ is not free.
\end{remark}

The ultimate goal of the present paper is the understand whether we can expect a positive answer on a (sub)question  devoted to the freeness of the full determinantal arrangement in $\mathbb{P}^{14}$. 

\begin{question}
\label{maintheorem}
Let us consider the following hypersurfaces arrangements $\mathcal{H} \, :\, V(F) = 0$  defined by $F=f_{1}f_{2}f_{3}f_{4}f_{5}f_{6}f_{7}f_{8}f_{9}f_{10}$. Is it true that $\mathcal{H}$ is free?
\end{question}

Towards approaching the above question, we study mid-step defined arrangements, namely those having the defining equation $Q_{k} = f_{1}f_{2}f_{3}f_{4}f_{5}f_{k}$ with $k \in \{6,7,8,9,10\}$. In particular, we can show the following results.

\begin{theorem}
\label{mid}
Let us consider the hypersurfaces arrangement
$$\mathcal{H}_{k} \, :\, V(Q_{k}) = 0$$
given by $Q_{k}=f_{1}f_{2}f_{3}f_{4}f_{5}f_{k}$ with $k \in \{6, 7, 8, 9\}$. Then $\mathcal{H}_{k}$ is free with the exponents $(\underbrace{1, ..., 1}_{13 \text{ times}},4)$.
\end{theorem}
\begin{corollary}
In the setting of the above theorem, one has
$${\rm reg}(S/J_{Q_{k}}) = 19$$
for each $k \in \{6,7,8,9\}$, so we reach an upper bound  for the regularity according to the content of \cite[Proposition 2.6]{BSDS}.
\end{corollary}
\begin{remark}
If we consider the arrangement $\mathcal{H}_{10}$ defined by $Q_{10}$, then it is not free since the minimal free resolution of the Milnor algebra has the following form:
$$0 \rightarrow S(-22)^{3} \rightarrow S^{5}(-21) \oplus S^{12}(-18) \rightarrow S^{15}(-17)\rightarrow S,$$
which is quite surprising.
\end{remark}

Our very ample numerical experiments suggest that the full determinantal arrangement $\mathcal{H} \, : f_{1} \cdots f_{10} = 0$ should be free with the exponents $(\underbrace{1, ..., 1}_{9 \text{ times}}, \underbrace{4, ..., 4}_{5 \text{ times}})$. In order to verify our claim we also checked other larger arrangements of hyperplanes, for instance we can verify that $\mathcal{C} \, : \, f_{1}f_{2}f_{3}f_{4}f_{7}f_{8}f_{9} = 0$ is free with the exponents $(\underbrace{1, ..., 1}_{12 \text{ times}}, 4,4)$. However, the derivations of degree $4$ seem to us that they do not have a natural geometric or algebraic explanation so it is very hard to find the basis of derivations for $\mathcal{H}$. We hope to solve this problem in the nearest future.
\section{Proofs}
We start with our proof of Theorem \ref{maintheorem}.
\begin{proof}
We are going to apply directly Saito's criterion. In order to do so, we need to find a basis of the $S$-modules ${\rm AR}(F_{j})$ for each $j \in \{5, ..., 10\}$. This means that in each case we need to find $14$ derivations for each ${\rm AR}(F_{j})$. Since for each choice of $F_{j}$ the procedure goes along the same lines, let us focus on the first case $F_{5} = f_{1}f_{2}f_{3}f_{4}f_{5}$.

We start with a group of (obvious to see) derivations, namely
\begin{center}
\begin{tabular}{l}
$\theta_{1} = z_{1}\cdot \partial_{x_{1}} + z_{2}\cdot \partial_{x_{2}} + z_{3}\cdot \partial_{x_{3}} + z_{4}\cdot\partial_{x_{4}}+z_{5}\cdot \partial_{x_{5}}$,\\
$\theta_{2} = z_{1}\cdot \partial_{y_{1}} + z_{2}\cdot \partial_{y_{2}} + z_{3}\cdot \partial_{y_{3}} + z_{4}\cdot\partial_{y_{4}}+z_{5}\cdot \partial_{y_{5}}$, \\
$\theta_{3} = y_{1}\cdot \partial_{x_{1}} + y_{2}\cdot \partial_{x_{2}} + y_{3}\cdot \partial_{x_{3}} + y_{4}\cdot\partial_{x_{4}} + y_{5}\cdot \partial_{x_{5}}$, \\
$\theta_{4} = y_{1}\cdot \partial_{z_{1}} + y_{2}\cdot \partial_{z_{2}} + y_{3}\cdot \partial_{z_{3}} + y_{4}\cdot\partial_{z_{4}} + y_{5}\cdot \partial_{z_{5}}$, \\
$\theta_{5} = x_{1}\cdot \partial_{y_{1}} + x_{2}\cdot \partial_{y_{2}} + x_{3}\cdot \partial_{y_{3}} + x_{4}\cdot\partial_{y_{4}}+x_{5}\cdot \partial_{y_{5}}$,\\
$\theta_{6} = x_{1}\cdot \partial_{z_{1}} + x_{2}\cdot \partial_{z_{2}} + x_{3}\cdot \partial_{z_{3}} + x_{4}\cdot\partial_{z_{4}}+x_{5}\cdot \partial_{z_{5}}$,\\
$\theta_{7} = x_{2}\cdot \partial_{x_{5}} + y_{2}\cdot \partial_{y_{5}} + z_{2}\cdot \partial_{z_{5}}$,\\
$\theta_{8} = x_{1}\cdot \partial_{x_{5}} + y_{1}\cdot \partial_{y_{5}} + z_{1}\cdot \partial_{z_{5}}$,\\
$\theta_{9} = y_{1}\cdot \partial_{y_{1}} + y_{2}\cdot \partial_{y_{2}} + y_{3}\cdot \partial_{y_{3}} + y_{4}\cdot \partial_{y_{4}} + y_{5}\cdot \partial_{y_{5}} -z_{1}\partial_{z_{1}} -z_{2}\partial_{z_{2}} -z_{3}\partial_{z_{3}} -z_{4}\partial_{z_{4}} - z_{5}\partial_{z_{5}}$.
\end{tabular}
\end{center}

We have additionally $5$ non-obvious-to-see relations among the partials derivatives (we have found them with use of \verb{Singular{ \cite{Singular}), namely:
\begin{center}
\begin{tabular}{l}
$\theta_{10} = 5x_{5} \cdot \partial_{x_{5}} + 5y_{5}\cdot \partial_{y_{5}} - z_{1}\cdot \partial_{z_{1}} - z_{2}\cdot \partial_{z_{2}} - z_{3}\cdot \partial_{z_{3}} - z_{4}\cdot\partial_{z_{4}} + 4z_{5}\cdot\partial_{z_{5}}$,\\
$\theta_{11} = 5x_{4}\cdot\partial_{x_{4}} + 5y_{4}\cdot \partial_{y_{4}} - 3z_{1}\cdot \partial_{z_{1}} - 3z_{2} \cdot \partial_{z_{2}} - 3z_{3} \cdot \partial_{z_{3}} + 2z_{4} \cdot \partial_{z_{4}} - 3z_{5} \cdot \partial_{z_{5}},$ \\
$\theta_{12} = 5x_{3}\cdot \partial_{x_{3}} -3y_{1}\cdot \partial_{y_{1}} - 3y_{2}\cdot \partial_{y_{2}} + 2y_{3}\cdot \partial_{y_{3}} - 3y_{4}\cdot \partial_{y_{4}} - 3y_{5}\cdot \partial_{y_{5}} + 5z_{3}\cdot \partial_{z_{3}},$\\
$\theta_{13} = 5x_{1} \cdot \partial_{x_{1}} + 5y_{1} \cdot \partial_{y_{1}} + z_{1}\cdot \partial_{z_{1}} - 4z_{2}\cdot \partial_{z_{2}} - 4z_{3}\cdot \partial_{z_{3}} - 4z_{4}\cdot \partial_{z_{4}} - 4z_{5}\cdot \partial_{z_{5}} ,$
\end{tabular}
\end{center}
and
\begin{align*}
\theta_{14} = 5x_{2}\cdot \partial_{x_{2}} -3y_{1}\cdot \partial_{y_{1}} + 2y_{2}\cdot\partial_{y_{2}} - 3y_{3}\cdot \partial_{y_{3}} - 3y_{4} \cdot \partial_{y_{4}} - 3y_{5} \cdot \partial_{y_{5}} - z_{1} \cdot \partial_{z_{1}} + 4z_{2} \cdot \partial_{z_{2}} \\ - z_{3} \cdot \partial_{z_{3}} - z_{4} \cdot \partial_{z_{4}} - z_{5} \cdot \partial_{z_{5}}. 
\end{align*}
Now we are going to construct Saito's matrix. In order to do so, let us write the coefficients of all $\theta_{i}$'s as the columns, and for the Euler derivation $E = \sum_{i=1}^{5}x_{i}\cdot \partial_{x_{i}} +  \sum_{j=1}^{5}y_{j}\cdot \partial_{y_{j}} +  \sum_{i=k}^{5}z_{k}\cdot \partial_{z_{k}}$ we write $r_{0} = (x_{1},..., x_{5}, y_{1}, ..., y_{5}, z_{1}, ...,z_{5})^{t}.$ \\
Then we get the following matrix
$$A = 
\left(\begin{tabu}{rrrrrrrrrrrrrrr}
x_1 & z_1 & 0 &  y_1 & 0&  0& 0&0&0&0&0& 5x_1 & 0& 0& 0 \\
x_2 & z_2 & 0&  y_2 & 0&  0&  0&    0&    0&    5x_2 &  0&  0&    0& 0& 0 \\
x_3 & z_3 & 0&  y_3 & 0&  0&  0&    0&    5x_3 & 0&     0&  0&    0& 0& 0 \\
x_4 & z_4 & 0&  y_4 & 0&  0&  0&    5x_4 & 0&    0&     0&  0&    0& 0& 0 \\
x_5 & z_5 & 0&  y_5 & 0&  0&  5x_5 & 0&    0&    0&    x_2 &  0&    x_1 & 0 & 0 \\
y_1 & 0& z_1 & 0&  y_1 & 0&  0&    0 &   -3y_1 &-3y_1 & 0& 5y_1 &  0& x_1 &0 \\
y_2 & 0& z_2 & 0&  y_2 & 0&  0&    0 &   -3y_2 & 2y_2 &  0&  0&    0& x_2 &0 \\
y_3 & 0& z_3 & 0&  y_3 & 0&  0&    0 &    2y_3 & -3y_3 & 0&  0&    0& x_3 &0 \\
y_4 & 0& z_4 & 0&  y_4 & 0&  0& 5y_4 &-3y_4 & -3y_4 & 0&  0&    0& x_4 &0 \\
y_5 & 0& z_5 & 0& y_5  & 0&5y_5 & 0 &-3y_5 & -3y_5 & y_2& 0&    y_1 & x_5 &0 \\
z_1 & 0& 0& 0& -z_1 & y_1 &-z_1 & -3z_1 &   0&-z_1 &   0&  z_1 &    0 &  0& x_1 \\
z_2 & 0& 0& 0& -z_2 & y_2 &-z_2 & -3z_2 &   0& 4z_2 & 0& -4z_2 &  0 &  0& x_2 \\
z_3 & 0& 0& 0& -z_3 & y_3 &-z_3 & -3z_3 &  5z_3 &-z_3 &   0&-4z_3 &  0 &  0& x_3 \\
z_4 & 0& 0& 0& -z_4 & y_4 &-z_4 &  2z_4 &  0& -z_4 &   0& -4z_4 &  0 &  0& x_4 \\
z_5 & 0& 0& 0& -z_5 & y_5 &4z_5 & -3z_5 &  0& -z_5 &  z_2 & -4z_5 & z_1 & 0& x_5
\end{tabu}\right).$$
After some cumbersome computations we obtain
$${\rm Det}(A) = 9375 \cdot F_{5},$$
which completes the proof.
\end{proof}

Now we are going to sketch the proof of Theorem \ref{mid}.
\begin{proof}
Once again, we are going to apply Saito's criterion. We focus on the case $k=7$ since other cases can be show in analogical way. The proof is heavily based on \verb{Singular{ computations and experiments. We can find polynomial derivations that preserves $\mathcal{H}$, namely
\begin{center}
\begin{tabular}{l}
$\theta_{1} = z_{1}\cdot \partial_{x_{1}} + z_{2}\cdot \partial_{x_{2}} + z_{3}\cdot \partial_{x_{3}} + z_{4}\cdot\partial_{x_{4}}+z_{5}\cdot \partial_{x_{5}}$,\\

$\theta_{2} = z_{1}\cdot \partial_{y_{1}} + z_{2}\cdot \partial_{y_{2}} + z_{3}\cdot \partial_{y_{3}} + z_{4}\cdot\partial_{y_{4}}+z_{5}\cdot \partial_{y_{5}}$, \\

$\theta_{3} = y_{1}\cdot \partial_{x_{1}} + y_{2}\cdot \partial_{x_{2}} + y_{3}\cdot \partial_{x_{3}} + y_{4}\cdot\partial_{x_{4}} + y_{5}\cdot \partial_{x_{5}}$, \\

$\theta_{4} = y_{1}\cdot \partial_{z_{1}} + y_{2}\cdot \partial_{z_{2}} + y_{3}\cdot \partial_{z_{3}} + y_{4}\cdot\partial_{z_{4}}+y_{5}\cdot \partial_{z_{5}}$,\\ 

$\theta_{5} = 3 x_{5}\cdot \partial_{x_{5}} + 3 y_{5}\cdot \partial_{y_{5}} - z_{1}\cdot \partial_{z_{1}} - z_{2}\cdot\partial_{z_{2}}-z_{3}\cdot \partial_{z_{3}}-z_{4}\cdot \partial_{z_{4}}+2 z_{5}\cdot \partial_{z_{5}}$,\\

$\theta_{6} = 2 x_{4}\cdot \partial_{x_{4}} + 2 y_{4}\cdot \partial_{y_{4}} - z_{1}\cdot \partial_{z_{1}}- z_{2}\cdot \partial_{z_{2}}- z_{3}\cdot \partial_{z_{3}}+ z_{4}\cdot \partial_{z_{4}} - z_{5}\cdot \partial_{z_{5}}$,\\

$\theta_{7} = 3 x_{3}\cdot \partial_{x_{3}} + 3 y_{3}\cdot \partial_{y_{3}} - 2 z_{1}\cdot \partial_{z_{1}}-2  z_{2}\cdot \partial_{z_{2}} + z_{3}\cdot \partial_{z_{3}}-2 z_{4}\cdot \partial_{z_{4}} -2  z_{5}\cdot \partial_{z_{5}}$,\\

$\theta_{8} = 6 x_{2}\cdot \partial_{x_{2}} + 6 y_{2}\cdot \partial_{y_{2}} - 5 z_{1}\cdot \partial_{z_{1}}+z_{2}\cdot \partial_{z_{2}} -5  z_{3}\cdot \partial_{z_{3}}-5 z_{4}\cdot \partial_{z_{4}} -5  z_{5}\cdot \partial_{z_{5}}$,\\

$\theta_{9} = x_{2}\cdot \partial_{x_{5}} + y_{2}\cdot \partial_{y_{5}}+z_{2}\cdot \partial_{z_{5}}$,\\

$\theta_{10} =3 x_{1}\cdot \partial_{x_{1}} +3 y_{1}\cdot \partial_{y_{1}}+z_{1}\cdot \partial_{z_{1}}-2 z_{2}\cdot \partial_{z_{2}}
-2 z_{3}\cdot \partial_{z_{3}}
-2 z_{4}\cdot \partial_{z_{4}}
-2 z_{5}\cdot \partial_{z_{5}}$,\\

$\theta_{11} = x_{1}\cdot \partial_{y_{1}} + x_{2}\cdot \partial_{y_{2}}+ x_{3}\cdot \partial_{y_{3}}+ x_{4}\cdot \partial_{y_{4}}+ x_{5}\cdot \partial_{y_{5}}$,\\

$\theta_{12} = x_{1}\cdot \partial_{z_{1}} + x_{2}\cdot \partial_{z_{2}}+ x_{3}\cdot \partial_{z_{3}}+ x_{4}\cdot \partial_{z_{4}}+ x_{5}\cdot \partial_{z_{5}}$,\\

\end{tabular}
\end{center}
\begin{align*}
\theta_{13} = y_{1}\cdot \partial_{y_{1}} + y_{2}\cdot \partial_{y_{2}} + y_{3}\cdot \partial_{y_{3}} + y_{4}\cdot\partial_{y_{4}} + y_{5}\cdot \partial_{y_{5}}- z_1 \cdot \partial_{z_{1}} - z_2 \cdot \partial_{z_{2}}- z_3 \cdot \partial_{z_{3}} &- z_4 \cdot \partial_{z_{4}} \\&- z_5 \cdot \partial_{z_{5}}, 
\end{align*}
and

\begin{table}[ht]
\begin{center}
\begin{tabularx}{0.96\textwidth}{lX}
$\theta_{14}=$         & $ 3  x_1  x_3  y_2  z_2 \cdot  \partial_{x_{2}}
+180  x_1  x_2  y_3  z_3 \cdot\partial_{x_{3}} 
  +(192x_1  x_2  y_4  z_3 - 9x_1  x_3  y_4  z_2 + 12x_1  x_3  y_2  z_4 -12x_1  x_2   y_3  z_4)\cdot \partial_{x_{4}}
  +(15x_1  x_3  y_5   z_2 - 12x_1  x_3  y_2  z_5)\cdot \partial_{x_{5}}
  +(3  x_3  y_1  y_2  z_2+60  x_2  y_1  y_2  z_3-60  x_1   y_2^2  z_3) \cdot\partial_{y_{2}}
   +(3  x_3  y_1  y_3  z_2-3  x_1  y_{3}^{2}  z_2 - 120 x_3  y_1  y_2  z_3\: +  180   x_2  y_1  y_3  z_3+120  x_1  y_2  y_3  z_3)\cdot\partial_{y_{3}}
   +(12  x_4  y_1  y_3  z_2 - 9 x_3  y_1  y_4   z_2-12  x_1  y_3  y_4  z_2 - 132  x_4  y_1  y_2  z_3+192  x_2  y_1  y_4  z_3+132  x_1  y_2   y_4  z_3+12  x_3  y_1  y_2  z_4-12  x_2  y_1  y_3  z_4)\cdot\partial_{y_{4}}
   +( 15x_3  y_1  y_5  z_2 -12  x_5  y_1  y_3  z_2 + 12  x_1  y_3  y_5  z_2 + 60  x_2  y_1  y_5  z_3-60  x_1  y_2  y_5  z_3-  12  x_3  y_1  y_2  z_5+12  x_2  y_1  y_3  z_5-12  x_1  y_2  y_3  z_5)\cdot~\partial_{y_{5}}
   +(4  x_1  y_3  z_2^2 -x_3  y_1  z_2^2 + 4  x_3  y_2  z_1  z_2-4  x_2  y_3  z_1  z_2+176  x_2  y_2  z_1  z_3  - 204  x_2  y_1  z_2  z_3+28  x_1  y_2  z_2  z_3)\cdot\partial_{z_{2}}
   +(204x_2   y_3  z_1  z_3 - 28x_3  y_2  z_1  z_3 - 24x_2  y_1  z_3^2+28  x_1  y_2  z_3^2+181  x_1  y_3  z_2  z_3-181  x_3   y_1  z_2  z_3)\cdot\partial_{z_{3}}
   +(8x_4  y_3  z_1  z_2 - 8x_3  y_4  z_1  z_2 - 40x_4  y_2  z_1  z_3 + 216  x_2  y_4  z_1  z_3 - 180x_4  y_1  z_2  z_3 + 180x_1  y_4  z_2  z_3 + 12x_3  y_2  z_1  z_4 
-12x_2  y_3  z_1  z_4-x_3  y_1  z_2  z_4-8  x_1  y_3  z_2  z_4-24  x_2  y_1  z_3  z_4+  40  x_1  y_2  z_3  z_4)\cdot\partial_{z_{4}}
   +(16  x_3  y_5  z_1  z_2 -16  x_5  y_3  z_1  z_2 - 16  x_5  y_2   z_1  z_3 + 192  x_2  y_5  z_1  z_3 - 72  x_5  y_1  z_2  z_3 + 12  x_1  y_5  z_2  z_3 - 12  x_3   y_2  z_1  z_5 + 12x_2  y_3  z_1  z_5 - x_3  y_1  z_2  z_5 + 4  x_1  y_3  z_2  z_5 - 132x_2    y_1  z_3  z_5 + 16  x_1  y_2  z_3  z_5)\cdot\partial_{z_{5}}.$
\end{tabularx}
\end{center}
\end{table}

We claim that the set $\{E, \theta_{1}, \theta_{2}, ..., \theta_{14}\}$
gives us a basis for $D(\mathcal{H})$. It is enough to observe that the determinant of Saito's matrix $A$ is equal to
$${\rm Det}(A) = 23328 \cdot Q_{7},$$
which completes the proof.
\end{proof}
\section{Further numerical experiments}
In order to understand better the geometry of determinantal hyperplane arrangements we decided to investigate all possible arrangements $\mathcal{C}$ given by triplets $F_{ijk} = f_{i}f_{j}f_{k}$ and given by $4$-tuples $F_{ijkl} = f_{i}f_{j}f_{k}f_{l}$ provided that the indices are pairwise distinct.
Our first observation is the following.
\begin{proposition}
Let $\mathcal{C} \subset \mathbb{P}^{14}_{\mathbb{C}}$ be a determinantal arrangement defined by the equation $F_{ijk} = f_{i}f_{j}f_{k}$, where $i,j,k \in \{1, ..., 10\}$ and the indices are pairwise distinct. Then $\mathcal{C}$ is never free.
\end{proposition}
\begin{proof}
Using a simple \verb{Singular{ routine we examined all choices of indices obtaining $120$ determinantal arrangements, and in each case ${\rm pdim}(S/J_{F_{ijk}}) >2$, which completes the proof.
\end{proof}
After that we focused on determinantal arrangements $\mathcal{C}$ given by $F_{ijkl} = f_{i}f_{j}f_{k}f_{l}$. We have exactly $210$ such arrangements, and among them we have exactly $5$ special arrangements, namely
\begin{itemize}
\item[a)] $\mathcal{C}_{1} \subset \mathbb{P}^{14}_{\mathbb{C}}$ given by $F_{1234}$,
\item[b)] $\mathcal{C}_{2} \subset \mathbb{P}^{14}_{\mathbb{C}}$ given by $F_{1567}$,
\item[c)] $\mathcal{C}_{3} \subset \mathbb{P}^{14}_{\mathbb{C}}$ given by $F_{2589}$,
\item[d)] $\mathcal{C}_{4} \subset \mathbb{P}^{14}_{\mathbb{C}}$ given by $F_{36810}$,
\item[e)] $\mathcal{C}_{5} \subset \mathbb{P}^{14}_{\mathbb{C}}$ given by $F_{47910}$.
\end{itemize}
These arrangements can be viewed as determinantal arrangements constructed as products of the maximal minors of appropriate generic $3\times 4$ matrix of indeterminantes. Thus by a result due to Buchweitz and Mond \cite{Buch} arrangements $\mathcal{C}_i$ with $i \in \{1,2,3,4,5\}$ are free.

Another important class of hypersurface arrangements was introduced by Bu\'se, Dimca, Schenck, and Sticlaru, and such arrangements are called \emph{nearly-free}.
\begin{definition}(\cite[Definition 2.6]{BSDS})
A reduced hypersurface $\mathcal{C} \subset \mathbb{P}^{n}_{\mathbb{C}}$ given by $F = 0$ is nearly-free if its Milnor algebra $M(F)$ admits a graded free resolution of the form
$$0\rightarrow S(-d_{n}-d) \rightarrow S(-d_{n}-d+1)\oplus\bigg(\oplus_{i=0}^{n-1}S(-d_{i}-d+1)\bigg) \rightarrow S^{n+1}(d+1)\rightarrow S$$
for some integers $d_{0} \leq d_{1} \leq d_{2} \leq \ldots d_{n}$.
\end{definition}
Next, we checked whether some of the remaining $205$ determinantal arrangements $\mathcal{C}$ given by $F_{ijkl}=0$ are nearly-free. It turns out that among $205$ arrangements we found $58$ having this peculiar property that their Milnor algebras $M(F_{ijkl})$ have the following minimal resolution:
$$0\rightarrow S(-15) \rightarrow S^{15}(-12) \rightarrow S^{15}(-11)\rightarrow S,$$
so these are not nearly-free arrangements, but to some extend these are close to them. Having a complete picture of the minimal resolution we can also calculate the regularity of $S/J_{F_{ijkl}}$ which is equal to
$${\rm reg}(S/J_{F_{ijkl}})=12.$$

\section*{Acknowledgments}
We would like to warmly thank Piotr Pokora for useful comments and suggestions.

Paulina Wi\'sniewska was partially supported by the National Science Center (Poland) Sonata Grant Nr \textbf{2018/31/D/ST1/00177}.

\vskip 0.5 cm

\bigskip
Marek Janasz,
Department of Mathematics,
Pedagogical University of Krakow,
ul. Podchorazych 2,
PL-30-084 Krak\'ow, Poland. \\
\nopagebreak
\textit{E-mail address:} \texttt{marek.janasz@up.krakow.pl}
\bigskip

Paulina Wi\'sniewska,
Department of Mathematics and Doctoral School,
Pedagogical University of Krakow,
ul. Podchorazych 2,
PL-30-084 Krak\'ow, Poland. \\
\nopagebreak
\textit{E-mail address:} \texttt{wisniewska.paulina.m@gmail.com}


\begin{thebibliography}{000}
\bibitem{Buch}
R.-O. Buchweitz and D. Mond, Singularities and Computer Algebra, London Math. Soc. Lecture Note Ser., vol. 324, Cambridge Univ. Press, Cambridge (2006), pp. 41-77.
\bibitem{BSDS}
L. Bus\'e, A. Dimca, H. Schenck, and G. Sticlaru, The Hessian polynomial and the Jacobian ideal of a reduced surface in $\mathbb{P}^{n}$. \textit{Adv. Math.} \textbf{392}: Article ID 108035, 22 p (2021).
\bibitem{Singular}
W.~Decker, G.-M. Greuel, G.~Pfister, and H.~Sch\"onemann,
\newblock {\sc Singular} {4-1-1} --- {A} computer algebra system for polynomial
  computations.
\newblock \url{http://www.singular.uni-kl.de}, 2018.
\bibitem{Dimca}
A. Dimca, \emph{Hyperplane arrangements. An introduction}. Universitext Cham: Springer. xii, 200 p. (2017).
\bibitem{ScTo}
H. Schenck and \c{S}tefan O. Toh\v{a}neanu, Freeness of Conic-Line Arrangements in $\mathbb P^2$. \textit{Comment. Math. Helv.} \textbf{84}: 235 -- 258 (2009).
\bibitem{Yim}
A. Yim, Homological properties of determinantal arrangements. \textit{J. Algebra} \textbf{471}: 220--239 (2017).

\end{thebibliography}
\end{document}